\documentclass[12pt, reqno]{amsart}
\usepackage{amsmath, amsthm, amscd, amsfonts, amssymb, graphicx, color}
\usepackage[bookmarksnumbered, colorlinks, plainpages]{hyperref}
\hypersetup{colorlinks=true,linkcolor=red, anchorcolor=green, citecolor=cyan, urlcolor=red, filecolor=magenta, pdftoolbar=true}

\usepackage{enumerate}
\usepackage{amssymb}
\usepackage{pst-node}
\usepackage{tikz-cd}

\newtheorem{theorem}{Theorem}[section]
\newtheorem{lemma}[theorem]{Lemma}
\newtheorem{proposition}[theorem]{Proposition}
\newtheorem{corollary}[theorem]{Corollary}
\theoremstyle{definition}

\newtheorem{problem}[theorem]{Problem}
\theoremstyle{remark}

\numberwithin{equation}{section}

\begin{document}

\setcounter{page}{1}

\title[On the unit sphere of positive operators]{On the unit sphere of positive operators}

\author[A.M. Peralta]{Antonio M. Peralta$^1$$^{*}$}

\address{$^{1}$Departamento de An{\'a}lisis Matem{\'a}tico, Facultad de
Ciencias, Universidad de Granada, 18071 Granada, Spain.}
\email{\textcolor[rgb]{0.00,0.00,0.84}{aperalta@ugr.es}}

%\address{$^{2}$Department of Pure Mathematics, Ferdowsi University of Mashhad, P. O. Box 1159, Mashhad 91775, Iran;
%\newline
%Tusi Mathematical Research Group (TMRG), Mashhad, Iran.}
%\email{\textcolor[rgb]{0.00,0.00,0.84}{second@afa.ac.ir}}

%\dedicatory{This paper is dedicated to Professor ABCD}

\let\thefootnote\relax\footnote{Copyright 2016 by the Tusi Mathematical Research Group.}



\date{Received: April 6th, 2018; Revised: yyyyyy; Accepted: zzzzzz.
\newline \indent $^{*}$Corresponding author}

\begin{abstract}
Given a C$^*$-algebra $A$, let $S(A^+)$ denote the set of those positive elements in the unit sphere of $A$. Let $H_1$, $H_2,$ $H_3$ and $H_4$ be complex Hilbert spaces, where $H_3$ and $H_4$ are infinite-dimensional and separable. In this note we prove a variant of Tingley's problem by showing that every surjective isometry $\Delta : S(B(H_1)^+)\to S(B(H_2)^+)$ or (respectively, $\Delta : S(K(H_3)^+)\to S(K(H_4)^+)$) admits a unique extension to a surjective complex linear isometry from $B(H_1)$ onto $B(H_2))$ (respectively, from $K(H_3)$ onto $B(H_4)$). This provides a positive answer to a conjecture posed by G. Nagy [\emph{Publ. Math. Debrecen}, 2018].
\end{abstract} \maketitle

\section{Introduction}

During the last thirty years, mathematicians have pursued an argument to prove or discard a positive solution to Tingley's problem (compare the survey \cite{Pe2018}). This problem, in which Geometry and Functional Analysis interplay, is just as attractive as difficult. The concrete statement of the problem reads as follows: Let $S(Y)$ and $S(Y)$ be the unit spheres of two normed spaces $X$ and $Y$, respectively. Suppose $\Delta : S(X) \to S(Y)$ is a surjective isometry. Does $\Delta $ admit an extension to a surjective real linear isometry from $X$ onto $Y$? \smallskip

A wide list of references, obtained during the last thirty years, encompasses positive solutions to Tingley's problem in the cases of sequence spaces \cite{Ding2002,Di:p,Di:8,Di:1}, spaces of measurable functions on a $\sigma$-finite measure space \cite{Ta:8, Ta:1, Ta:p}, spaces of continuous functions \cite{Wang}, finite-dimensional polyhedral spaces \cite{KadMar2012}, finite-dimensional C$^*$-algebras \cite{Tan2016,Tan2017,Tan2017b}, $K(H)$ spaces \cite{PeTan16}, spaces of trace class operators \cite{FerGarPeVill17}, and $B(H)$ spaces \cite{FerPe17,FerPe17b,FerPe17c}. The most recent achievements in this line establish that a surjective isometry between the unit spheres of two arbitrary von Neumann algebras admits a unique extension to a surjective real linear isometry between the corresponding von Neumann algebras \cite{FerPe17d}, and an excellent contribution due to M. Mori contains a complete positive solution to Tingley's problem for surjective isometries between the unit spheres of von Neumann algebra preduals \cite{Mori2017}. Readers interested in learning more details can consult the recent survey \cite{Pe2018}.\smallskip

The particular setting of C$^*$-algebras, and specially the von Neumann algebra $B(H),$ of all bounded linear operators on a complex Hilbert space $H$, and its hermitian subalgebras and subspaces, offer the optimal conditions to consider an interesting variant to Tingley's problem. Let us introduce some notation first. If $B$ is a subset of a Banach space $X$, we shall write $S(B)$ for the intersection of $B$ and $S(X)$. Given a C$^*$-algebra $A$, the symbol $A^+$ will denote the cone of positive elements in $A$, while $S(A^+)$ will stand for the sphere of positive norm-one operators.

\begin{problem}\label{problem positive Tingley's problem} Let $\Delta : S(A^+)\to S(B^+)$ be a surjective isometry, where $A$ and $B$ are C$^*$-algebras. Does $\Delta$ admit an extension to a surjective complex linear isometry $T : A\to B$?
\end{problem}

The hypothesis in Problem \ref{problem positive Tingley's problem} are certainly weaker than the hypothesis in Tingley's problem. However, the required conclusion is also weaker, because the goal is to find a surjective linear isometry $T:A\to B$ satisfying $T|_{S(A^+)} \equiv \Delta,$ and we not care about the behavior of $T$ on the rest of $S(A)$. For the moment being, both problems seem to be independent.\smallskip

Problem \ref{problem positive Tingley's problem} can be also considered when $A$ and $B$ are replaced with the space $(C_p(H), \|\cdot\|_p)$ of all $p$-Schatten-von Neumann operators ($1\leq p\leq \infty$). For a finite-dimensional complex Hilbert space $H$ and $p\geq 1$,  L. Moln{\'a}r and G. Nagy determined all surjective isometries on the space $(S(C_1(H)^+), \|.\|_p)$ (see \cite[Theorem 1]{MolNag2012}). Problem \ref{problem positive Tingley's problem} has been solved by L. Moln{\'a}r and W. Timmermann for the space $C_1(H)$ of trace class operators on an arbitrary complex Hilbert space $H$  (see \cite[Theorem 4]{MolTim2003}). Given $p$ in the interval $(1,\infty)$ and $A=B= C_p(H)$, a complete solution to Problem \ref{problem positive Tingley's problem} has been obtained by G. Nagy in \cite[Theorem 1]{Nagy2013}.\smallskip

Following the usual notation, for each complex Hilbert space $H$, we identify $C_\infty (H)$ with the space $B(H)$. In a very recent contribution, G. Nagy resumes the study of Problem \ref{problem positive Tingley's problem} for $B(H)$. Applying deep geometric arguments in spectral theory and projective geometry, Nagy solves this problem in the case in which $H$ is finite-dimensional. Concretely, if $H$ is a finite-dimensional complex Hilbert space, and $\Delta : S(B(H)^+)\to S(B(H)^+)$ is an isometry, then $\Delta$ is surjective and there exists a surjective complex linear isometry $T : B(H) \to B(H)$ satisfying $T(x) = \Delta(x)$ for all $x\in B(H)$ (see \cite[Theorem]{Nagy2017}). In the third section of \cite{Nagy2017}, Nagy conjectures that an infinite-dimensional version of his result holds true for surjective isometries on $S(B(H)^+)$.\smallskip

In this paper we present a argument to prove Nagy's conjecture. Concretely, in Theorem \ref{t positive Tigley for B(H)} we prove that for any two complex Hilbert spaces $H_1$ and $H_2$, every surjective isometry $\Delta : S(B(H_1)^+)\to S(B(H_2)^+)$ can be extended to a surjective complex linear isometry {\rm(}actually, a $^*$-isomorphism or a $^*$-anti-automorphism{\rm)} $T: B(H_1)\to B(H_2)$.\smallskip

A closer look at the technical arguments in recent papers dealing with Tingley's problem (compare, for example, \cite{Tan2016, Tan2017, Tan2017b, PeTan16, FerPe17b, FerPe17c}, and \cite{FerPe17d}) reveals a common strategy based on a geometric tool asserting that a surjective isometry between the unit spheres of two Banach spaces $X$ and $Y$ preserves maximal convex sets of the corresponding spheres (see \cite[Lemma 5.1$(ii)$]{ChenDong2011}, \cite[Lemma 3.5]{Tan2014}). This is a real obstacle in our setting, because this geometric tool is not applicable for a surjective isometry $\Delta : S(B(H_1)^+)\to S(B(H_2)^+)$ where we can hardly identify a surjective isometry between the unit spheres of two normed spaces. We shall develop independent arguments to prove the Nagy's conjecture. In this note we introduce new arguments built upon a recent abstract characterization of those elements in $S(B(H)^+)$ which are projections in terms of their distances to positive elements in $S(B(H)^+)$ (see \cite{Pe2018b}), and the Bunce-Wright-Mackey-Gleason theorem (see \cite[Theorem A]{BuWri92} or \cite[Theorem A]{BuWri94}).\smallskip

In section \ref{sec:K(H)} we also give a positive solution to Problem \ref{problem positive Tingley's problem} in the case in which $A$ and $B$ are spaces of compact operators on separable complex Hilbert spaces (see Theorem \ref{t Nagy for K(ell2)}). In this final section, the Bunce-Wright-Mackey-Gleason theorem will be replaced with a theorem due to J.F. Aarnes which guarantees the linearity of quasi-states on $K(H)$ (see \cite{Aarnes70})

\section{Basic background and precedents}

In the recent note \cite{Pe2018b} we establish a geometric characterization of those element in the unit sphere of an atomic von Neumann algebra $M$ (or in the unit sphere of the space of compact operators on a separable complex Hilbert space) in terms of the unit sphere of positive operators around an element. Let us recall the basic definitions. Let $E$ and $P$ be subsets of a Banach space $X$. We define the \emph{unit sphere around $E$ in $P$} as the set $$Sph(E;P) :=\left\{ x\in P : \|x-b\|=1 \hbox{ for all } b\in E \right\}.$$ If $x$ is an element in $X$, we write $Sph(x;P)$ for $Sph(\{x\};P)$. If $E$ is a subset of a C$^*$-algebra $A$, we shall write $Sph^+ (E)$ or $Sph_A^+ (E)$ for the set $Sph(E;S(A^+))$. For each element $a$ in $A$, we shall write $Sph^+ (a)$ instead of $Sph^+ (\{a\})$.\smallskip

We recall that a non-zero projection $p$ in a C$^*$-algebra $A$ is called minimal if $p A p = \mathbb{C} p$. A von Neumann algebra $M$ is called atomic if it coincides with the weak$^*$ closure of the linear span of its minimal projections. It is known that for every atomic von Neumann algebra $M$ there exists a family $\{H_i\}_{i}$ of complex Hilbert spaces such that $\displaystyle M = \bigoplus_j^{\ell_{\infty}} B(H_{j})$ (compare \cite[\S 2.2]{S} or  \cite[\S V.1]{Tak}). Every projection $p$ in an atomic von Neumann algebra $M$ is the least upper bound of the set of all minimal projections in $M$ which are smaller than or equal to $p$.\smallskip

Let $a$ be a positive norm-one element in an atomic von Neumann algebra $M$. In \cite[Theorem 2.3]{Pe2018b} we prove that $$ \hbox{$a$ is a projection } \Leftrightarrow Sph^+_{M} \left( Sph^+_{M}(a) \right) =\{a\}.$$ This holds true when $M=B(H)$. Theorem 2.5 in \cite{Pe2018b} assures that the same equivalence remains true for any positive element $a$ in the unit sphere of $K(H_2)$, where $H_2$ is a separable complex Hilbert space. Since, for every $E\subseteq S(A^+)$,  the set $Sph_A^+ (E)$ is completely determined by the metric structure of $S(A^+)$, the next results borrowed from \cite{Pe2018b} are direct consequences of the characterizations just commented. We recall first that, for a C$^*$-algebra $A$, the symbol $\mathcal{P}roj(A)$ will denote the set of all projections in $A$, and $\mathcal{P}roj(A)^*$ will stand for $\mathcal{P}roj(A)\backslash\{0\}$.

\begin{corollary}\label{c first consequence}{\rm\cite[Corollary 2.6]{Pe2018b}} Let $\Delta : S(M^+)\to S(N^+)$ be a surjective isometry, where $M$ and $N$ are atomic von Neumann algebras. Then $\Delta$ maps $\mathcal{P}roj(M)^*$ onto $\mathcal{P}roj(N)^*$, and the restriction $\Delta|_{\mathcal{P}roj(M)^*} : \mathcal{P}roj(M)^*\to \mathcal{P}roj(N)^*$ is a surjective isometry.
\end{corollary}

\begin{corollary}\label{c first consequence K(H)}{\rm\cite[Corollary 2.7]{Pe2018b}} Let $H_2$ and $H_3$ be separable complex Hilbert spaces, and let us assume that $\Delta : S(K(H_2)^+)\to S(K(H_3)^+)$ is a surjective isometry. Then $\Delta$ maps $\mathcal{P}roj(K(H_2))^*$ to $\mathcal{P}roj(K(H_3))^*$, and the restriction $$\Delta|_{\mathcal{P}roj(K(H_2))^*} : \mathcal{P}roj(K(H_2))^*\to \mathcal{P}roj(K(H_3))^*$$ is a surjective isometry.$\hfill\Box$
\end{corollary}

Along this note, the closed unit ball and the dual space of a Banach space $X$ will be denoted by $\mathcal{B}_{_X}$ and $X^*$, respectively. The symbol $X^{**}$ will stand for the second dual space of $X$. Given a subset $B\subset X,$ we shall write $\mathcal{B}_{_B}$ for $\mathcal{B}_{_X}\cap B$. The shall write $A_{sa}$ for the self-adjoint part of a C$^*$-algebra $A$, while the symbol  $(A^*)^+$ will stand for the set of positive functionals on $A$. If $A$ is unital, $\textbf{1}$ will stand for its unit.\smallskip

Suppose $a$ is a positive element in the unit sphere of a von Neumann algebra $M$. The \emph{range projection} of $a$ in $M$ (denoted by $r(a)$) is the smallest projection $p$ in $M$ satisfying $a p = a$. It is known that the sequence $\left( (1/n \textbf{1}+a)^{-1} a \right)_n$ is monotone increasing to $r(a)$, and hence it converges to $r(a)$ in the weak$^*$-topology of $M$. Actually, $r(a)$ also coincides with the weak$^*$-limit of the sequence $(a^{1/n})_{n}$ in $M$  (see \cite[2.2.7]{Ped}). It is also known that the sequence $(a^{n})_{n}$ converges to a projection $s(a)=s_{_{M}}(a)$ in $M,$ which is called the \emph{support projection} of $a$ in $M$. Let us observe that  the support projection of a norm-one element in $M$ might be zero, however, for each positive element $a$ in the unit sphere of the bidual space of a C$^*$-algebra $A$ we have $  s_{_{A^{**}}} (a) \neq 0$ (compare \cite[$(2.3)$]{Pe2018b}).\smallskip

We recall next some known properties in C$^*$-algebra theory. Let $p$ be a projection in a unital C$^*$-algebra $A$. Suppose that $x\in S(A)$ satisfies $pxp =p,$ then \begin{equation}\label{eq Peirce 2 norm one}\hbox{$x = p + (\textbf{1}-p) x (\textbf{1}-p)$, \ \ }
 \end{equation} (see, for example, \cite[Lemma 3.1]{FerPe18}). Suppose that $b\in A^+$ satisfies $pbp =0,$  then \begin{equation}\label{eq pbp =0 implies p perp b for p positive}\hbox{ $p b = b p =0$,}
 \end{equation} (see \cite[$(2.2)$]{Pe2018b}). If $p$ is a non-zero projection in a C$^*$-algebra $A$, and $a$ is an element in $S(A^+)$ satisfying $p\leq a$ then \begin{equation}\label{eq new 0711} a = p + (\textbf{1}-p) a (\textbf{1}-p),\end{equation} (see \cite[$(2.4)$]{Pe2018b}).\smallskip

\section{Surjective isometries between normalized positive elements of type I von Neumann factors}\label{sec:B(H)}

Along this section $H_1$ and $H_2$ will be two complex Hilbert spaces. The main goal here is to determine when a surjective isometry $\Delta : S(B(H_1)^+)\to S(B(H_2)^+)$ can be extended to a surjective complex linear isometry from $B(H_1)$ onto $B(H_2)$. The case in which $H_1=H_2$ with dim$(H_1)<\infty$ has been positively solved by G. Nagy in \cite{Nagy2017}. In the just quoted reference, Nagy conjectures that the same statement holds true when $H$ is infinite-dimensional. The previous Corollary \ref{c first consequence} gives a generalization of \cite[Claim 1]{Nagy2017} for arbitrary complex Hilbert spaces. Our next aim is to provide a proof of the whole conjecture posed by Nagy.\smallskip

We recall next a tool that will be used throughout the rest of the paper.  Henceforth, let the symbol $\ell_2^n$ stand for an $n$-dimensional complex Hilbert space. If $p$ is a rank-one projection in $B(\ell_2^2)$, up to an appropriate representation, we can assume that $p = \left(
                                                                                                     \begin{array}{cc}
                                                                                                       1 & 0 \\
                                                                                                       0 & 0 \\
                                                                                                     \end{array}
                                                                                                   \right)$.
Given $t\in [0,1]$ the element $q_t = \left(
                                                                                                     \begin{array}{cc}
                                                                                                       t & \sqrt{t(1-t)} \\
                                                                                                       \sqrt{t(1-t)} & 1-t \\
                                                                                                     \end{array}
                                                                                                   \right)$
also is a projection in $B(\ell_2^2)$ and $\|p - q_t\|= \sqrt{1-t}$. Therefore, for each non-trivial projection $p$ in $B(\ell_2^2)$ we can find another non-trivial projection $q$ in $B(\ell_2^2)$ with $0< \|p-q\|<1$. Similar arguments show that if $H$ is a complex Hilbert space with dim$(H)\geq 2$, for each non-trivial projection $p$ in $B(H)$ we can find another non-trivial projection $q$ in $B(H)$ with $0< \|p-q\|<1$.\smallskip

Let $A$ and $B$ be C$^*$-algebras. A linear map $\Phi: A\to B$ is called a Jordan $^*$-homomorphism if $\Phi (a^*) = \Phi(a)^*$ and $\Phi (a \circ b) = \Phi (a) \circ \Phi (b)$ for all $a,b\in A$.\smallskip

Elements $a,b$ in a C$^*$-algebra $A$ are called orthogonal (written $a\perp b$) if $a b^* = b^* a =0$. It is known that $ \|a+ b\| =\max\{\|a\|,\|b\| \},$ for every $a,b\in A$ with $a\perp b$. Clearly, self-adjoint elements $a,b$ in $A$ are orthogonal if and only if $a b =0$. \smallskip

The following technical result will be needed for latter purposes.

\begin{lemma}\label{l orthomorphism preserves orthogonality} Suppose $\Delta : \mathcal{P}roj(B(H_1))\to \mathcal{P}roj(B(H_2))$ is a {\rm(}unital{\rm)} isometric order automorphism, where $H_1$ and $H_2$ are complex Hilbert spaces. Then $\Delta$ preserves orthogonality, that is, $\Delta (p) \Delta(q)=0$ whenever $p q =0$ in $\mathcal{P}roj(M)$. Furthermore, the same conclusion holds for an isometric order automorphism $\Delta : \mathcal{P}roj(K(H_1))\to \mathcal{P}roj(K(H_2))$.
\end{lemma}

\begin{proof} Let $e_1$ and $v_1$ be orthogonal minimal projections in $B(H_1)$. By hypothesis $\Delta(e_1)$ and $\Delta(v_1)$ are minimal projections, and $\Delta (e_1+v_1)$ is a projection with $\Delta(e_1+v_1) \geq \Delta(e_1),\Delta(v_1)$. Since $\| \Delta(e_1) - \Delta(v_1) \| =\|e_1-v_1\|=1,$ \cite[Lemma 2.1]{Pe2018b} assures the existence of a minimal projection $\widehat{e}\in B(H_2)^{**}$ such that one of the following statements holds:
\begin{enumerate}[$(a)$]\item $\widehat{e}\leq \Delta(e_1)$ and $\widehat{e}\perp \Delta(v_1)$ in $B(H_2)^{**}$;
\item $\widehat{e}\leq \Delta(v_1)$ and $\widehat{e}\perp \Delta(e_1)$ in $B(H_2)^{**}$.
\end{enumerate} Having in mind that $\Delta(e_1)$ and $\Delta(v_1)$ are minimal projections in $B(H_2)^{**}$ the above statements are equivalent to  \begin{enumerate}[$(a)$]\item $\widehat{e}= \Delta(e_1)$ and $\widehat{e}\perp \Delta(v_1)$ in $B(H_2)^{**}$, and hence $\Delta(e_1)\perp \Delta(v_1)$;
\item $\widehat{e}= \Delta(v_1)$ and $\widehat{e}\perp \Delta(e_1)$ in $B(H_2)^{**}$, and hence $\Delta(e_1)\perp \Delta(v_1)$.
\end{enumerate}

Now let us take two arbitrary projections $p,q\in B(H_1)$ with $p q=0$. We pick two arbitrary minimal projections $\widehat{e}_1\leq \Delta(p)$ and $\widehat{v}_1\leq \Delta(p)$. By hypothesis, there exist minimal projections $e_1$, $v_1$ in $B(H_1)$ satisfying $\Delta(e_1) =  \widehat{e}_1$, $\Delta(v_1) = \widehat{v}_1$, $e_1\leq p$ and $v_1\leq q$. The condition $pq=0$ implies $ e_1 v_1=0$. Applying the conclusion in the first paragraph we deduce that $\Delta(e_1) =  \widehat{e}_1\perp \Delta(v_1) = \widehat{v}_1$. We have therefore proved that $\widehat{e}_1\perp  \widehat{v}_1$ whenever $\widehat{e}_1$ and $\widehat{v}_1$ are minimal projections with $\widehat{e}_1\leq \Delta(p)$ and $\widehat{v}_1\leq \Delta(p)$. Since in $B(H_2)$ the projection $\Delta(p)$ (respectively, $\Delta(q)$) is the least upper bound of all minimal projections in $B(H_2)$ which are smaller than or equal to $\Delta(p)$ (respectively, $\Delta(q)$) it follows that $\Delta(p)\perp \Delta (q)$.\smallskip

If $\Delta : \mathcal{P}roj(K(H_1))\to \mathcal{P}roj(K(H_2))$ is an isometric order automorphism the conclusion follows with similar arguments.
\end{proof}

In 1951, R.V. Kadison proved that a surjective linear isometry $T$ from a unital C$^*$-algebra $A$ onto another C$^*$-algebra $B$ is of the form $T =u \Phi$, where $u$ is a unitary element in $B$ and $\Phi$ is a Jordan $^*$-isomorphism from $A$ onto $B$ (see \cite[Theorem 7]{Kad51}, see also \cite{Patt69}). In particular every unital surjective linear isometry $T: A\to B$ is a Jordan $^*$-isomorphism. Furthermore, if $A$ is a factor von Neumann algebra, then $T$ is a $^*$-isomorphism or a $^*$-anti-isomorphism. In our next result we begin with weaker hypotheses.\smallskip

\begin{proposition}\label{p first consequence bis} Let $\Delta : S(B(H_1)^+)\to S(B(H_2)^+)$ be a surjective isometry, where $H_1$ and $H_2$ are complex Hilbert spaces. Then $\Delta$ maps $\mathcal{P}roj(B(H_1))^*$ to $\mathcal{P}roj(B(H_2))^*$, and the restriction $\Delta|_{\mathcal{P}roj(B(H_1))^*} : \mathcal{P}roj(B(H_1))^*\to \mathcal{P}roj(B(H_2))^*$ is a surjective isometry and a unital order automorphism. We further know that $\Delta|_{\mathcal{P}roj(B(H_1))^*}$ preserves orthogonality.\smallskip

Consequently, if $T: B(H_1)\to B(H_2)$ is a bounded complex linear mapping such that $T(S(B(H_1)^+)) = S(B(H_2)^+)$ and $T|_{S(B(H_1)^+)} : S(B(H_1)^+) \to  S(B(H_2)^+)$ is an isometry, then $T$ is a $^*$-isomorphism or a $^*$-anti-automorphism.
\end{proposition}

\begin{proof} Most part of the first statement is given by Corollary \ref{c first consequence}. Following an idea outlined by G. Nagy in \cite[Proof of Claim 2]{Nagy2017}, we shall begin by proving that $\Delta$ is unital. By Corollary \ref{c first consequence}, $\Delta(\textbf{1})$ is a non-zero projection. We recall that $\textbf{1}$ is the unique non-zero projection in $B(H_2)$ whose distance to any other projection is $0$ or $1$. If $\Delta(\textbf{1}) = q_0\neq \textbf{1}$, there exists a non-zero projection $q_1\in B(H_2)$ such that $0<\| q_1 - q_0\| = \| \Delta(\textbf{1}) - q_1\|<1.$ A new application of Corollary \ref{c first consequence} to $\Delta^{-1}$ implies the existence of a non-zero projection $p_1\in B(H_1)$ such that $\Delta (p_1) = q_1$. In this case we have, $p_1\neq \textbf{1}$ and $1=\|\textbf{1}-p_1\| =\| \Delta(\textbf{1}) - \Delta (p_1) \| = \| q_0 -   q_1\|<1,$ witnessing a contradiction.\smallskip

Let us prove next that $\Delta|_{\mathcal{P}roj(B(H_1))^*}$ is an order automorphism. To this aim, let us pick $p,q\in {\mathcal{P}roj(B(H_1))^*}$ with $p\leq q$. Let $v$ be a minimal projection in $B(H_2)$ such that $v\leq \textbf{1}-\Delta(q)= \Delta(\textbf{1}) -\Delta(q)$. The element $z= v +\frac12 (\textbf{1}-v)$ lies in $S(B(H_2)^+)$. Pick $x\in S(B(H_1)^+)$ satisfying $\Delta(x) = z$. Since $$\frac12=\| z -\textbf{1} \|=\| \Delta (x) - \Delta(\textbf{1}) \| = \|x-\textbf{1}\|,$$ we deduce that $x$ is invertible. Furthermore, since $$1\geq \|x-q\|=\| \Delta (x) - \Delta (q) \|=\| z - \Delta (q) \| \geq \| v  (z - \Delta (q)) v \| = \|v\|=1.$$ By Lemma 2.1 in \cite{Pe2018b} there exists a minimal projection $e$ in $B(H_1)^{**}$ such that one of the following statements holds:
\begin{enumerate}[$(a)$]\item $e\leq x$ and $e\perp q$ in $B(H_1)^{**}$;
\item $e\leq q$ and $e\perp x$ in $B(H_1)^{**}$.
\end{enumerate} Case $(b)$ is impossible because $x$ is invertible in $B(H_1)$ (and hence in $B(H_1)^{**}$). Therefore $e\leq x$ and $e\perp q$, which implies that $e\perp p$, because $p\leq q$. Therefore, \cite[Lemma 2.1]{Pe2018b} implies that $1=\| x- p \| =\|\Delta(x) - \Delta(p) \| =\|z-\Delta (p) \|$. A new application of \cite[Lemma 2.1]{Pe2018b} assures the existence of a minimal projection $w$ in $B(H_2)^{**}$ such that one of the following statements holds:
\begin{enumerate}[$(a)$]\item $w\leq z$ and $w\perp \Delta(p)$ in $B(H_2)^{**}$;
\item $w\leq \Delta(p)$ and $w\perp z$ in $B(H_2)^{**}$.
\end{enumerate} As before, case $(b)$ is impossible because $z$ is invertible in $B(H_2)$. Therefore $w\leq z = v +\frac12 (\textbf{1}-v)$ and $w\perp \Delta(p)$. It can be easily deduced from the minimality of $w$ in $B(H_2)^{**}$ and the minimality of $v$ in $B(H_2)$ that $v= w\perp \Delta(p)$. We have therefore shown that $\Delta(p)$ is orthogonal to every minimal projection $v$ in $B(H_2)$ with $v\leq \textbf{1}-\Delta(q)$, and consequently $\textbf{1}-\Delta(q) \leq \textbf{1}-\Delta(p)$, or equivalently, $\Delta(p)\leq \Delta(q)$.\smallskip

The statement affirming that $\Delta|_{\mathcal{P}roj(B(H_1))^*}$ preserves orthogonality can be derived from Lemma \ref{l orthomorphism preserves orthogonality}.\smallskip

To prove the final statement, let $T: B(H_1)\to B(H_2)$ be a linear mapping such that $T(S(B(H_1)^+)) = S(B(H_2)^+)$ and $T|_{S(B(H_1)^+)} : S(B(H_1)^+) \to  S(B(H_2)^+)$ is an isometry. By applying the conclusion of the first statement, we deduce that $T|_{S(B(H_1)^+)}$ maps $\mathcal{P}roj(B(H_1))^*$ onto $\mathcal{P}roj(B(H_2))^*$, and the restricted mapping $T|_{\mathcal{P}roj(B(H_1))^*} : \mathcal{P}roj(B(H_1))^*\to \mathcal{P}roj(B(H_2))^*$ is a surjective isometry and a unital order automorphism. Clearly, $T$ preserves projections and orthogonality among them (just observe that the sum of two projections is a projection if and only if they are orthogonal). Since every hermitian element in a von Neumann algebra can be approximated in norm by a finite real linear combination of mutually orthogonal projections (see \cite[Proposition 1.3.1]{S}), and by the above properties $T(a^2) = T(a)^2$ and $T(a) = T(a)^*$, whenever $a$ is a finite real linear combination of mutually orthogonal projections, we deduce that $T(b^2) = T(b)^2$ and $T(b)^* = T(b)$ for every hermitian element $b$ in $B(H_1)$. It is well known that this is equivalent to say that $T$ is a Jordan $^*$-isomorphism. The rest follows from \cite[Corolary 11]{Kad51} because $B(H_1)$ is a factor.
\end{proof}

%We include next an abstract version of \cite[Lemma 2]{Nagy2017}.
%
%\begin{lemma}\label{l distance to smaller projections is constant} Let $a$ be a positive element in a von Neumann algebra $M$. Suppose that $\|a - p\|$ is constant for every projection $p$ in $M\backslash \{0,\textbf{1}\}$. Then $a$ is a scalar multiple of the identity.
%\end{lemma}
%
%\begin{proof} Let $\{e(\lambda)\}_{-\infty <\lambda<\infty} =\{ r((\lambda \textbf{1} -a)^+)\}_{-\infty <\lambda<\infty}$ be the system of projections or the spectral measure appearing in the spectral resolution of the element $a$ in the sense of \cite[Theorem 1.11.3]{S}. Accordingly to this spectral resolution  $\displaystyle a = \int_{-\|a\|}^{\|a\|+0} \lambda de(\lambda),$ where the integral employed in \cite[Theorem 1.11.3]{S} is an abstract Radon-Stieltjes integral with respect to the strong$^*$-topology of $M$. Given $0\leq \mu \leq \|a\|$, the projection $\displaystyle p_{\mu} = \int_{\mu}^{\|a\|+0} de(\lambda)= \textbf{1}-e(\lambda)$ lies in $M$ and $\|a-p_{\mu}\|= 1-\lambda$. It follows from our hypothesis that $p_{\mu}$ must be zero for every $\mu\in (0,\|a\|]$, and hence $e(\lambda)=0$ for every $\lambda<\|a\|,$ and $e(\lambda) = \textbf{1}$ for every $\lambda>\|a\|$. Therefore $a =\|a\| \textbf{1}$ is a scalar multiple of the identity. The conclusion could be also obtained by considering the von Neumann subalgebra of $M$ generated by $a$ and the unit element, which is identified with $C(K)$ for some hyper-Stonean space $K$, and applying the same hypothesis in $C(K)$.
%\end{proof}

We continue with an analogue of \cite[Claim 3]{Nagy2017}.

\begin{lemma}\label{l delta preserves supports} Let $\Delta : S(B(H_1)^+)\to S(B(H_2)^+)$ be a surjective isometry, where $H_1$ and $H_2$ are complex Hilbert spaces. Let $p_0,p_1,\ldots, p_m$ be mutually orthogonal projections with $\displaystyle \sum_{k=0}^m p_k =\textbf{1}$, and let $\lambda_1,\ldots,
\lambda_m$ be real numbers in the interval $(0,1)$. Then $\displaystyle s_{_{B(H_2)}} \left(\Delta \left(p_0+\sum_{k=1}^m \lambda_k p_k \right)\right) = \Delta(p_0)$.
\end{lemma}

\begin{proof} Set $\displaystyle a= p_0+\sum_{k=1}^m \lambda_k p_k.$ Since $\Delta (\textbf{1}) =\textbf{1}$ and $\|\Delta(a) - \textbf{1} \|= \|\Delta(a) - \Delta (\textbf{1}) \|= \| a-\textbf{1}\| = \max\{ 1- \lambda_k : k =1,\ldots, m\}<1,$ we deduce that $a$ and $\Delta(a)$ both are invertible elements.\smallskip

Let $\widehat{v}$ be a minimal projection in $B(H_2)$.  By Proposition \ref{p first consequence bis}, there exists a minimal projection $v$ in $B(H_1)$ satisfying $\Delta (v) = \widehat{v}$. By the hypothesis on $\Delta$ and  Proposition \ref{p first consequence bis}, we have $\| a - (\textbf{1}-v)\| = 1$ if and only if $\| \Delta(a) - \Delta(\textbf{1}-v)\| =\| \Delta(a) - (\textbf{1}-\Delta(v))\| = 1$. Combining the invertibility of $a$ and $\Delta (a)$, and the minimality of $v$ and $\Delta (v)$ with Lemma 2.1 in \cite{Pe2018b}, we deduce that $$v\leq p_0 \Leftrightarrow  v\leq a \Leftrightarrow \| a - (\textbf{1}-v)\| = 1 \Leftrightarrow \| \Delta(a) - (\textbf{1}-\Delta(v))\| = 1\Leftrightarrow \Delta(v) \leq \Delta (a).$$ Therefore, a minimal projection $v$ satisfies $v\leq p_0$ if and only if $v\leq a$ if and only if $\Delta(v) \leq \Delta (a)$ if and only if $\Delta(v) \leq \Delta (p_0)$.\smallskip

Take a minimal projection $\widehat{v}\in B(H_2)$ such that $\widehat{v}=\Delta(v)\leq \Delta(p_0)$. We know from the above that $\widehat{v} \leq \Delta (a),$ and ${v}\leq a$. Since in $B(H_2)$ every projection $q$ is the least upper bound of all minimal projections $\widehat{v}$ with $\widehat{v}\leq q$, we deduce that $\Delta(p_0)\leq \Delta(a),$ and hence $ \Delta(p_0) \leq s_{_{B(H_2)}} (\Delta (a ))$. Another application of the above property shows that $\widehat{v}\leq \Delta(p_0)$ for every minimal projection $\widehat{v}\in B(H_2)$ with $\widehat{v}\leq s_{_{B(H_2)}} (\Delta (a ))\leq \Delta(a)$. Therefore $s_{_{B(H_2)}} (\Delta (a ))=\Delta(p_0)$.
\end{proof}

Accordingly to the usual notation, given a C$^*$-algebra $A$, the symbol $S(\hbox{Inv} (A)^+)$ will denote the set of all positive invertible elements in $S(A)$. A projection $p$ in a unital C$^*$-algebra $A$ will be called \emph{co-minimal} if $1-p$ is a minimal projection in $A$. The symbol $\hbox{co-min-}\mathcal{P}roj(A)$ will stand for the set of all co-minimal projections in $A$.

\begin{theorem}\label{t main technical theorem with bispherical for co-minimal} Let $a$ be an invertible element in $S(B(H)^+)$, where $H$ is an infinite-dimensional complex Hilbert space. Suppose that $s_{_{B(H)}} (a)\neq 0$. Then the following statements hold:
\begin{enumerate}[$(a)$]\item $Sph(a; \hbox{co-min-}\mathcal{P}roj(B(H))) =\{ p\in \hbox{co-min-}\mathcal{P}roj(B(H)) : \textbf{1}-p \leq s_{_{B(H)}} (a)  \};$
\item The identity $$Sph(Sph(a; \hbox{co-min-}\mathcal{P}roj(B(H))); S(\hbox{Inv} (B(H))^+)) $$ $$=\{ x\in S(\hbox{Inv} (B(H))^+) :  s_{_{B(H)}} (a) \leq x \}$$ holds.
\end{enumerate}
\end{theorem}

\begin{proof}$(a)$ Let $v$ be a minimal projection in $B(H)$.  Combining the invertibility of $a$, and the minimality of $v$ with \cite[Lemma 2.1]{Pe2018b} it can be seen that $$ v\leq a \Leftrightarrow \| a - (\textbf{1}-v)\| = 1.$$ Therefore, for each minimal projection $v$ in $B(H)$ we have \begin{equation}\label{eq new 1511} v\leq s_{_{B(H)}} (a)\leq a \hbox{ if and only if } \|a-(\textbf{1}-v)\| =1,
\end{equation} (compare \eqref{eq new 0711}).\smallskip

$(\supseteq)$ Take $p\in \hbox{co-min-}\mathcal{P}roj(B(H))$ with $\textbf{1}-p \leq s_{_{B(H)}} (a)$. Applying \eqref{eq new 1511} with $v= \textbf{1}-p$ we get $\| a - p\| =1$.\smallskip

$(\subseteq)$ Take now $p\in \hbox{co-min-}\mathcal{P}roj(B(H))$ with $\| a- (\textbf{1}-(\textbf{1}-p))\|=\| a - p\| =1$. We deduce from \eqref{eq new 1511} that $\textbf{1}-p \leq s_{_{B(H)}} (a)\leq a$.\smallskip

$(b)$ $(\supseteq)$  Let us take $x\in S(\hbox{Inv} (B(H))^+)$ satisfying $s_{_{B(H)}} (a) \leq x$. For each $p \in \hbox{co-min-}\mathcal{P}roj(B(H)))$ with $\| a-p\| =1$, we know from $(a)$ that $\textbf{1}-p \leq s_{_{B(H)}} (a)\leq x$. Applying the statement in \eqref{eq new 0711} we have $\textbf{1}-p \leq s_{_{B(H)}} (x)$. A new application of $(a)$ to the element $x$ gives $\| x - p \| =1$. This shows that $x$ lies in $Sph(Sph(a; \hbox{co-min-}\mathcal{P}roj(B(H))); S(\hbox{Inv} (B(H))^+))$.\smallskip

$(\subseteq)$ Take $x\in S(\hbox{Inv} (B(H))^+)$ satisfying $\|x-p\| = 1$ for every projection $p$ in $Sph(a; \hbox{co-min-}\mathcal{P}roj(B(H)))$. Applying $(a)$, it can be seen that, for every minimal projection $v$ in $B(H)$ with $v\leq s_{_{B(H)}} (a)$ we have $$\textbf{1}-v\in Sph(a; \hbox{co-min-}\mathcal{P}roj(B(H))),$$ and hence $\|x-(\textbf{1}-v) \|=1$. Since $x \in S(\hbox{Inv} (B(H))^+)$ and $v$ is minimal, it follows from $(a)$ that $v \leq s_{_{B(H)}} (x)$. We have proved that $v \leq s_{_{B(H)}} (x)\leq x$ whenever $v$ is a minimal projection with $v\leq s_{_{B(H)}} (a)$. Therefore $s_{_{B(H)}} (a)\leq x.$
\end{proof}

The next lemma is a simple observation.

\begin{lemma}\label{l positive invertible are preserved} Let $\Delta : S(A^+)\to S(B^+)$ be a surjective isometry, where $A$ and $B$ are unital C$^*$-algebras. Suppose $\Delta(\textbf{1}) =\textbf{1}$. Then $\Delta (S(\hbox{Inv} (A)^+)) = S(\hbox{Inv} (B)^+)$.
\end{lemma}

\begin{proof} We observe that an element $b\in S(A^+)$ is invertible if and only if the inequality $\|a-\textbf{1}\|<1$ holds. Therefore $b\in S(\hbox{Inv} (A)^+)$ if and only if $\| b-\textbf{1}\| <1$ if and only if $\| \Delta(b)-\Delta(\textbf{1})\|= \| \Delta(b)-\textbf{1} \| <1$ if and only if $\Delta (b) \in S(\hbox{Inv} (B)^+)$.
\end{proof}

We are now in position to establish the main result of this section, which proves the conjecture posed by G. Nagy in \cite[\S 3]{Nagy2017}.

\begin{theorem}\label{t positive Tigley for B(H)} Let $\Delta : S(B(H_1)^+)\to S(B(H_2)^+)$ be a surjective isometry, where $H_1$ and $H_2$ are complex Hilbert spaces. Then there exists a surjective complex linear isometry {\rm(}actually, a $^*$-isomorphism or a $^*$-anti-automorphism{\rm)} $T: B(H_1)\to B(H_2)$ satisfying $\Delta (x) = T(x)$ for all $x\in S(B(H_1)^+)$.
\end{theorem}

\begin{proof} Proposition \ref{p first consequence bis} implies that $$\Delta|_{\mathcal{P}roj(B(H_1))^*} : \mathcal{P}roj(B(H_1))^*\to \mathcal{P}roj(B(H_2))^*$$ is a surjective isometry and a unital order automorphism.\smallskip

If dim$(H_1)$ is finite, it can be easily seen from the above that dim$(H_1)=$dim$(H_2)$, just observe that dim$(H)$($<\infty$) is precisely the cardinality of every maximal set of minimal projections in $B(H)$. In this case, the desired conclusion was established by G. Nagy in \cite[Theorem]{Nagy2017}.\smallskip

Let us assume that $H_1$ is infinite-dimensional. We define a vector measure $\mu: \mathcal{P}roj(B(H_1)) \to B(H_2)$ given by $\mu (0) =0$ and $\mu (p) = \Delta (p)$ for all $p$ in $\mathcal{P}roj(B(H_1))^*$. It is clear that $\mu (p) \in \mathcal{P}roj(B(H_2))$ for every $p$ in $\mathcal{P}roj(B(H_1))$. In particular \begin{equation}\label{eq measure mu is bounded} \{ \|\mu (p)\| : p\in\mathcal{P}roj(B(H_1)) \} =\{0,1\}.
\end{equation}

We claim that $\mu$ is finitely additive, that is \begin{equation}\label{eq mu is finitely additive} \mu \left( \sum_{j=1}^m p_j \right) = \sum_{j=1}^m \mu (p_j), \end{equation} for every family $\{p_1,\ldots,p_m\}$ of mutually orthogonal projections in $B(H_1)$. Namely, we can assume that $p_j\neq 0$ for every $j$. Lemma \ref{l orthomorphism preserves orthogonality} and Proposition \ref{p first consequence bis} assure that $\{\Delta(p_1),\ldots,\Delta(p_m)\}$ are mutually orthogonal projections in $B(H_2)$. We also know from Proposition \ref{p first consequence bis} that $\displaystyle \mu \left( \sum_{j=1}^m p_j \right) = \Delta \left( \sum_{j=1}^m p_j \right)$ and $\mu (p_j)= \Delta (p_j)$ are projections in $B(H_2)$ with $\displaystyle\mu \left( \sum_{j=1}^m p_j \right) = \Delta \left( \sum_{j=1}^m p_j \right)\geq \mu (p_j)= \Delta (p_j)$ for all $j\in \{1,\ldots, m\},$ and hence  $\displaystyle\mu \left( \sum_{j=1}^m p_j \right) \geq \sum_{j=1}^m \mu(p_j).$ Since $\displaystyle \sum_{j=1}^m \mu(p_j)$ and $\displaystyle \sum_{j=1}^m p_j$ are the least upper bounds of $\{\Delta(p_1),\ldots,\Delta(p_m)\}$ and $\{p_1,\ldots,p_m\}$ in $B(H_2)$ and $B(H_1)$, respectively, and $\Delta|_{\mathcal{P}roj(B(H_1))^*}$ is an order isomorphism (see Proposition \ref{p first consequence bis}), we get $\displaystyle\mu \left( \sum_{j=1}^m p_j \right) = \sum_{j=1}^m \mu(p_j).$\smallskip

We have therefore shown that $\mu$ is a bounded finitely additive measure. We are in position to apply the Bunce-Wright-Mackey-Gleason theorem (see \cite[Theorem A]{BuWri92} or \cite[Theorem A]{BuWri94}), and thus there exists a unique bounded complex linear operator $T : B(H_1)\to B(H_2)$ satisfying \begin{equation}\label{syntesis of T} \hbox{$T(p ) = \mu (p) = \Delta (p)$ for every $p\in \mathcal{P}roj(B(H_1))^*$.}
 \end{equation}Since $T|_{\mathcal{P}roj(B(H_1))^*} = \Delta|_{\mathcal{P}roj(B(H_1))^*} : \mathcal{P}roj(B(H_1))^*\to \mathcal{P}roj(B(H_2))^*$ is a surjective isometry and a unital order automorphism, the second part in Proposition \ref{p first consequence bis} implies that $T$ is a surjective isometry and a $^*$-isomorphism or a $^*$-anti-isomorphism.\smallskip

It only remains to prove that $T(x) = \Delta(x)$ for every $x\in S(B(H_1))$. Let us begin with an element of the form $\displaystyle a = p_0 + \sum_{j=1}^m \lambda_j p_j$, where $\lambda_j\in \mathbb{R}^+$, and $p_0,p_1,\dots, p_m$ are mutually orthogonal non-zero projections in $B(H_1)$ with $\displaystyle \sum_{j=0}^m  p_j=\textbf{1}.$\smallskip

Since $\Delta(\textbf{1}) =\textbf{1}$, Lemma \ref{l positive invertible are preserved} assures that $\Delta (S(\hbox{Inv} (B(H_1))^+)) = S(\hbox{Inv} (B(H_2))^+)$. Furthermore, since the sets $Sph(a; \hbox{co-min-}\mathcal{P}roj(B(H_1)))$ and $$Sph(Sph(a; \hbox{co-min-}\mathcal{P}roj(B(H_1))); S(\hbox{Inv} (B(H_1)^+))$$ are determined by the norm, the element $a$, the set $S(\hbox{Inv} (B(H_1))^+)$, and the set $Sph(a; \hbox{co-min-}\mathcal{P}roj(B(H_1))),$ and all these structures are preserved by $\Delta$, we deduce that
$$ \Delta(Sph(a; \hbox{co-min-}\mathcal{P}roj(B(H_1)))) = Sph(\Delta(a) ; \hbox{co-min-}\mathcal{P}roj(B(H_2))),$$ and \begin{equation}\label{eq Delta preserves double sph comin} \Delta\left( Sph(Sph(a; \hbox{co-min-}\mathcal{P}roj(B(H_1))); S(\hbox{Inv} (B(H_1)^+)) \right)
 \end{equation} $$= Sph(Sph(\Delta(a); \hbox{co-min-}\mathcal{P}roj(B(H_2))); S(\hbox{Inv} (B(H_2)^+)).$$

Lemma \ref{l delta preserves supports} implies that $s_{_{B(H_2)}} \left(\Delta(a)\right) = \Delta(p_0)$. We have already commented that $\Delta (a)$ is invertible (compare Lemma \ref{l positive invertible are preserved}).\smallskip

Now applying Theorem \ref{t main technical theorem with bispherical for co-minimal}$(b)$ we deduce that $$Sph(Sph(a; \hbox{co-min-}\mathcal{P}roj(B(H_1))); S(\hbox{Inv} (B(H_1)^+))$$ $$=\{ x\in S(\hbox{Inv} (B(H_1))^+) :  s_{_{B(H)}} (a)=p_0 \leq x \} $$ $$= p_0 + \{y \in (\textbf{1}-p_0) B(H_1)^+ (\textbf{1}-p_0) : y\in \hbox{Inv}((\textbf{1}-p_0) B(H_1) (\textbf{1}-p_0)), \ \|y\|\leq 1 \} $$ $$= p_0 + \mathcal{B}_{_{Inv((\textbf{1}-p_0) B(H_1)^+ (\textbf{1}-p_0))}} = p_0 + \mathcal{B}_{_{Inv( B((\textbf{1}-p_0)(H_1))^+)}},$$ and $$ Sph(Sph(\Delta(a); \hbox{co-min-}\mathcal{P}roj(B(H_2))); S(\hbox{Inv} (B(H_2)^+)) $$ $$ = \Delta(p_0) + \mathcal{B}_{_{Inv( B((\textbf{1}-\Delta(p_0))(H_2))^+)}}.$$ To simplify the notation, let us denote $K_1 = (\textbf{1}-p_0)(H_1)$ and $K_2 = (\textbf{1}-\Delta(p_0))(H_2)$. By combining the above identities with \eqref{eq Delta preserves double sph comin} we can consider the following diagram of surjective isometries:
\begin{equation}\label{diagram} \begin{tikzcd} p_0 + \mathcal{B}_{_{Inv( B(K_1)^+)}} \arrow[swap]{d}{\tau_{-p_0}} \arrow{r}{\Delta} &  \Delta(p_0) + \mathcal{B}_{_{Inv( B(K_2)^+)}}  \\
  \mathcal{B}_{_{Inv( B(K_1)^+)}} \arrow[dashrightarrow]{r}{\Delta_a} &  \mathcal{B}_{_{Inv( B(K_2)^+)}} \arrow[u, "\tau_{\Delta(p_0)}"]
\end{tikzcd}
 \end{equation} where, $\tau_{z}$ denotes the translation by $z$, and $\Delta_a$ is the surjective isometry making the above diagram commutative.\smallskip

Let us observe the following property: for each unital C$^*$-algebra $A$, the set $\mathcal{B}_{_{Inv( A^+)}},$ of all positive invertible elements in the closed unit ball of $A,$ is a convex subset with non-empty interior in $A_{sa}$. Actually, if $a,b\in \mathcal{B}_{_{Inv( A^+)}}$ we know that $t a +(1-t) b \in \mathcal{B}_{_{A^+}}$ for every $t\in [0,1]$ (see \cite[Theorem 1.4.2]{S}). By the invertibility of $a,b$ we can find positive constants $m_1,m_2$ such that $m_1 \textbf{1}\leq a$ and $m_2 \textbf{1} \leq b$. Therefore, $(t m_1 + (1-t) m_2) \textbf{1} \leq  t a +(1-t) b,$ which guarantees that $t a +(1-t) b$ is invertible too. We note that the open unit ball in $A_{sa}$ with center $\frac12 \textbf{1}$ and radius $\frac12$ is contained in $\mathcal{B}_{_{Inv( A^+)}}$. Since $\Delta_a : \mathcal{B}_{_{Inv( B(K_1)^+)}}\to \mathcal{B}_{_{Inv( B(K_2)^+)}}$ is a surjective isometry, we are in position to apply Manckiewicz´s theorem (see \cite[Theorem 5 and Remark 7]{Mank1972}) to deduce the existence of a surjective real linear isometry $T_a : B(K_1)_{sa}\to B(K_2)_{sa}$ and $z_0\in B(K_2)_{sa}$ such that \begin{equation}\label{eq identity with Delta after Mank thm} \hbox{$\Delta_a (x) = T_a (x)+z_0,$ for all $x \in \mathcal{B}_{_{Inv( B(K_1)^+)}}.$}
\end{equation}

Since $\Delta (\textbf{1}) = \textbf{1}$, it follows from the construction above that $\Delta_{a} (\textbf{1}_{_{B(K_1)}}) = \textbf{1}_{_{B(K_2)}},$ and thus  $T_{a} (\textbf{1}_{_{B(K_1)}}) +z_0 = \textbf{1}_{_{B(K_2)}}.$ \smallskip

Let us recall that an element $s$ in $B(K_2)_{sa}$ is called a symmetry if $s^2 =1$. Actually every symmetry in $B(K_2)_{sa}$ is of the form $s = p_1 - (\textbf{1}_{_{B(K_2)}}-p_1)$, where $p_1$ is a projection. The real Jordan Banach algebras $B(K_1)$ and $B(K_2)$ (equipped with the natural Jordan product $x\circ y = \frac12 (x y + yx)$) are prototypes of JB-algebras in the sense employed in \cite{WriYoung1978} and \cite{IsRod1995}. Since $T_a : B(K_1)_{sa}\to B(K_2)_{sa}$ is a surjective isometry, by applying \cite[Theorem 1.4]{IsRod1995}, we deduce the existence of a central symmetry $s\in B(K_2)_{sa}$, and a unital Jordan $^*$-isomorphism $\Phi_a : B(K_1)_{sa}\to B(K_2)_{sa}$ such that $T_a (x) = s \Phi_a (x)$, for all $x\in B(K_1)_{sa}$. However, the unique central symmetries in $B(K_2)_{sa}$ are $\textbf{1}_{_{B(K_2)}}$ and $-\textbf{1}_{_{B(K_2)}}$. Summing up we have $$ \textbf{1}_{_{B(K_2)}} -z_0 =T_a (\textbf{1}_{_{B(K_1)}}) = s \textbf{1}_{_{B(K_2)}} =s = \pm  \textbf{1}_{_{B(K_2)}}.$$ Then, one and only one of the next statements holds:
\begin{enumerate}[$(1)$] \item $z_0 = 0,$ and thus $T_{a} (\textbf{1}_{_{B(K_1)}}) = \textbf{1}_{_{B(K_2)}},$ and $T_a$ is a Jordan $^*$-isomorphism;
\item $z_0 = 2 \ \textbf{1}_{_{B(K_2)}}= 2 (\textbf{1}-\Delta(p_0)),$ and thus $T_{a} (\textbf{1}_{_{B(K_1)}}) = - \textbf{1}_{_{B(K_2)}},$ and $\Phi_a =- T_a$ is a Jordan $^*$-isomorphism;
\end{enumerate}

We claim that case $(2)$ is impossible, otherwise, by inserting the element $p_0 + \frac12 (\textbf{1}-p_0)$ (where $\frac12 \textbf{1}_{_{B(K_1)}} \equiv \frac12 (\textbf{1}-p_0) \in \mathcal{B}_{_{Inv( B(K_1)^+)}}\cong \mathcal{B}_{_{Inv( B((\textbf{1}-p_0)(H_1))^+)}}$) in the diagram \eqref{diagram} (see also \eqref{eq identity with Delta after Mank thm}) we get $$\Delta\left(p_0 + \frac12 (\textbf{1}-p_0)\right) = \Delta(p_0) + \Delta_a \left(\frac12 (\textbf{1}-p_0) \right) = \Delta(p_0) + T_a \left(\frac12 (\textbf{1}-p_0) \right) + z_0 $$ $$= \Delta(p_0)+ 2 \ (\textbf{1}- \Delta(p_0)) - \frac12 \Phi_a \left( (\textbf{1}-p_0) \right) = \Delta(p_0)+ 2 \ (\textbf{1}- \Delta(p_0)) - \frac12 \ (\textbf{1}- \Delta(p_0)) $$ $$= \Delta(p_0)+ \frac32 \ (\textbf{1}- \Delta(p_0)),$$ which proves that $\frac32 = \| \Delta(p_0)+ \frac32 \ (\textbf{1}- \Delta(p_0))\|= \|\Delta(p_0 + \frac12 (\textbf{1}-p_0))\| =1$, leading to a contradiction.\smallskip

Therefore, only case $(1)$ holds, and hence $T_a$ is a Jordan $^*$-isomorphism.\smallskip

We shall prove next that \begin{equation}\label{eq Ta and Delta coincide on projections in the orthogonal} \Delta (q) = T_a (q), \hbox{ for every projection } q \leq \textbf{1}-p_0.
\end{equation} Namely, take a projection $q \leq \textbf{1}-p_0.$ By inserting the element $b= p_0 + q +\frac12 (\textbf{1} -q-p_0)$ in the diagram \eqref{diagram} (see also \eqref{eq identity with Delta after Mank thm}) we get $$\Delta (b) = \Delta\left(  p_0 + q +\frac12 (\textbf{1} -q-p_0) \right) =\Delta (p_0) + \Delta_a \left( q +\frac12 (\textbf{1} -q-p_0) \right)  $$ $$ =\Delta (p_0) + T_a \left( q +\frac12 (\textbf{1} -q-p_0) \right)=\Delta (p_0) + T_a \left(q \right) + \frac12 T_a(\textbf{1} -q-p_0),$$ which assures that  $s_{_{B(H_2)}} (\Delta (b) ) = \Delta (p_0) + T_a \left(q \right)$. On the other hand, Lemma \ref{l delta preserves supports} implies that $s_{_{B(H_2)}} (\Delta (b) ) =  \Delta (s_{_{B(H_2)}} (b) ) = \Delta (p_0 +q ) = \hbox{(by \eqref{eq mu is finitely additive})} = \Delta (p_0) + \Delta (q)$. We have therefore shown that $ \Delta (p_0) + T_a \left(q \right) = \Delta (p_0) + \Delta (q),$ which concludes the proof of \eqref{eq Ta and Delta coincide on projections in the orthogonal}.\smallskip

Now, inserting our element $\displaystyle a = p_0 + \sum_{j=1}^m \lambda_j p_j$ (where $\lambda_j\in \mathbb{R}^+$, and $p_0,p_1,\dots, p_m$ are mutually orthogonal non-zero projections in $B(H_1)$ with $\displaystyle \sum_{j=0}^m  p_j=\textbf{1}$) in \eqref{diagram} (see also \eqref{eq identity with Delta after Mank thm}) we deduce that $$\Delta(a) = \Delta\left(  p_0 + \sum_{j=1}^m \lambda_j p_j \right) =\Delta (p_0) + \Delta_a \left( \sum_{j=1}^m \lambda_j p_j \right) =\Delta (p_0) + T_a \left( \sum_{j=1}^m \lambda_j p_j \right)  $$ $$=\Delta (p_0) +  \sum_{j=1}^m \lambda_j T_a \left(p_j \right) =\hbox{(by \eqref{eq Ta and Delta coincide on projections in the orthogonal})} =\Delta (p_0) +  \sum_{j=1}^m \lambda_j \Delta \left(p_j \right) $$ $$= \hbox{(by \eqref{syntesis of T})} =T (p_0) +  \sum_{j=1}^m \lambda_j T \left(p_j \right) = T(a).$$

Finally it is well known that every element in the unit sphere of $B(H_1)$ can be approximated in norm by elements of the form $\displaystyle a = p_0 + \sum_{j=1}^m \lambda_j p_j,$ where $\lambda_j\in \mathbb{R}^+$, and $p_0,p_1,\dots, p_m$ are mutually orthogonal non-zero projections in $B(H_1)$ with $\displaystyle \sum_{j=0}^m  p_j=\textbf{1}$. Therefore, since $\Delta$ and $T$ are continuous and coincide on elements of the previous form, we deduce that $\Delta (x) = T(x)$, for every $x\in S(B(H_1)^+)$, which concludes the proof.
\end{proof}

\section{Surjective isometries between normalized positive elements of compact operators}\label{sec:K(H)}

Throughout this section $H_3$ and $H_4$ will denote two separable infinite-dimen-sional complex Hilbert spaces. Our goal here will consist in studying surjective isometries $\Delta :S( K( H_{3} )^{+})\to S( K( H_{4} )^{+} ).$\smallskip

We begin with a technical result.

\begin{lemma}\label{l surjective isometries between balls of positive} Let $\Delta : \mathcal{B}_{B(H_1)^+}\to \mathcal{B}_{B(H_2)^+}$ be a surjective isometry, where $H_1$ and $H_2$ are complex Hilbert spaces. Suppose that $\Delta(\mathcal{P}roj(B(H_1))) = \mathcal{P}roj(B(H_2))$. Then there exists a surjective complex linear isometry {\rm(}actually a Jordan $^*$-iso-morphism{\rm)} $T : B(H_1)\to B(H_2)$ such that one of the next statements holds:
\begin{enumerate}[$(a)$] \item $\Delta(x) = T(x),$ for all $x\in \mathcal{B}_{B(H_1)^+}$;
\item $\Delta(x) = \textbf{1}-T(x),$ for all $x\in \mathcal{B}_{B(H_1)^+}$.
\end{enumerate}
Furthermore, since $B(H_1)$ and $B(H_2)$ are factors we can also deduce that $T$ is a $^*$-isomorphism or a $^*$-anti-isomorphism.
\end{lemma}

\begin{proof} We consider the real Banach spaces $B(H_1)_{sa}$ and $B(H_2)_{sa}$ as JB-algebras in the sense employed in \cite{WriYoung1978}. The proof is heavily based on a deep result due to P. Mankiewicz asserting that every bijective isometry between convex sets in normed linear spaces with nonempty interiors, admits a unique extension to a bijective affine isometry between the corresponding spaces (see \cite[Theorem 5 and Remark 7]{Mank1972}). Let us observe that $\mathcal{B}_{B(H_1)^+}\subset \mathcal{B}_{B(H_1)_{sa}}$ and $\mathcal{B}_{B(H_2)^+}\subset \mathcal{B}_{B(H_2)_{sa}}$ are convex sets with nonempty interiors (just observe that the open unit ball in $B(H)_{sa}$ of radius $1/2$ and center $\frac12 \textbf{1}$, is contained in $\mathcal{B}_{B(H)^+}$). Thus, by Mankiewicz's theorem, there exists a bijective real linear isometry $T: B(H_1)_{sa}\to B(H_2)_{sa}$ and $z_0\in \mathcal{B}_{B(H_2)^+}$ such that $\Delta(x) = T(x) +z_0$, for all $x\in \mathcal{B}_{B(H_1)^+}$. We denote by the same symbol $T$ the bounded complex linear operator from $ B(H_1)$ to $B(H_2)$ given by $T(x+i y ) = T(x) + i T(y)$ for all $x,y \in B(H_1)_{sa}$.\smallskip

On the other hand, since, by hypothesis, $\Delta$ preserves projections, we infer that $z_0$ is a projection and $T(\mathcal{P}roj(B(H_1)))+ z_0= \Delta(\mathcal{P}roj(B(H_1))) = \mathcal{P}roj(B(H_2)).$ The projections $0$ and $\textbf{1}$ are the unique projections in $B(H_1)$ (or in $B(H_2)$) whose distance to another projection is $0$ or $1$. If $z_0=\Delta(0)\neq 0,\textbf{1}$, then there exists a non-trivial projection $q$ in $B(H_2)$ satisfying $0< \| \Delta (0)-q \| <1$. This implies that $$\{0,1\}\ni \| 0- \Delta^{-1} (q) \| = \| \Delta(0) - q \| \in (0,1),$$ which is impossible. We have therefore proved that $z_0= \Delta (0) \in\{ 0,\textbf{1}\}.$ Similar arguments show that $\Delta(\textbf{1}) = T(\textbf{1})+z_0\in \{0,\textbf{1}\}$. Applying that $\Delta$ is a bijection we deduce that precisely one of the next statements holds:
\begin{enumerate}[$(a)$] \item $\Delta(0)=z_0=0$ and $\Delta(\textbf{1})=\textbf{1}$;
\item $\Delta(0)=z_0=\textbf{1}$ and $\Delta(\textbf{1})=0$.
\end{enumerate}

If $z_0 = \Delta (0) = 0,$ and $\Delta(\textbf{1}) = T(\textbf{1}) +z_0=\textbf{1}$,  the mapping $T: B(H_1)_{sa} \to B(H_2)_{sa}$ is a unital and surjective real linear isometry between JB-algebras. Applying \cite[Theorem 4]{WriYoung1978}, we deduce that $T$ is a Jordan isomorphism. In particular, the complex linear extension $T: B(H_1)\to B(H_2)$ is a complex linear Jordan $^*$-isomorphism and $\Delta(x)= T(x),$ for all $x\in \mathcal{B}_{B(H_1)^+}$. We arrive to statement $(a)$ in our conclusion.\smallskip

If $\Delta(0)=z_0=\textbf{1}$ and $\Delta(\textbf{1})=T(\textbf{1}) +z_0=0$, we have $T(\textbf{1}) = -\textbf{1}$. Therefore $-T: B(H_1)_{sa} \to B(H_2)_{sa}$ is a unital and surjective real linear isometry. The arguments in the previous case prove that the complex linear extension of $-T$, denoted by $-T: B(H_1)\to B(H_2),$ is a complex linear Jordan $^*$-isomorphism and $\Delta(x)= \textbf{1}-(-T(x)),$ for all $x\in \mathcal{B}_{B(H_1)^+}$. We have therefore arrived to statement $(b)$ in our conclusion.\smallskip

The last statement follows from Corollary 11 in \cite{Kad51}.
\end{proof}

Corollary \ref{c first consequence K(H)} admits an strengthened version which was established in \cite{Pe2018b}.

\begin{theorem}\label{t bi spherical set K(H)}{\rm\cite[Theorem 2.8]{Pe2018b}} Let $H_2$ be a separable infinite-dimensional complex Hilbert space. Then the identity $$Sph^+_{K(H_2)} \left( Sph^+_{K(H_2)}(a) \right) =\left\{ b\in S(K(H_2)^+) : \!\! \begin{array}{c}
    s_{_{K(H_2)}} (a) \leq s_{_{K(H_2)}} (b),  \hbox{ and }\\
     \textbf{1}-r_{_{B(H_2)}}(a)\leq \textbf{1}-r_{_{B(H_2)}}(b)
  \end{array}\!\!
 \right\},$$ holds for every $a$ in the unit sphere of $K(H_2)^+$. $\hfill\Box$
\end{theorem}

We can now improve the conclusion of Corollary \ref{c first consequence K(H)}.

\begin{proposition}\label{p second consequence K(H)} Let $H_3$ and $H_4$ be separable complex Hilbert spaces. Let us assume that $H_3$ is infinite-dimensional.
Let $\Delta : S(K(H_3)^+)\to S(K(H_4)^+)$ be a surjective isometry. Then the following statements hold:
\begin{enumerate}[$(a)$]\item $\Delta$ preserves projections, that is, $\Delta (\mathcal{P}roj(K(H_3))^*) = \mathcal{P}roj(K(H_4))^*$, and the restricted mapping $\Delta|_{\mathcal{P}roj(K(H_3))^*} : \mathcal{P}roj(K(H_3))^*\to \mathcal{P}roj(K(H_4))^*$ is a surjective isometry and an order automorphism. Furthermore, $\Delta (p) \Delta (q)=0$ for every $p,q\in \mathcal{P}roj(K(H_3))^*$ with $p q =0$;
\item For every finite family $p_1,\ldots, p_n$ of mutually orthogonal minimal projections in $K(H_3)$, and $1=\lambda_1\geq \lambda_2,\ldots, \lambda_n\geq 0$ we have $$\Delta\left(\sum_{j=1}^{n} \lambda_j p_j \right)= \sum_{j=1}^{n} \lambda_j \Delta\left( p_j \right).$$
\end{enumerate}
\end{proposition}

\begin{proof} $(a)$ The first part of the statement has been proved in Corollary \ref{c first consequence K(H)}. We shall show next that $\Delta$ preserves order between non-zero projections.\smallskip

We claim that given $p,e_1\in \mathcal{P}roj(K(H_3))^*$ with $e_1$ minimal and $e_1\perp p$ we have \begin{equation}\label{eq projection and some minimal orthogonal} \Delta(p + e_1) \geq \Delta(p).
\end{equation}

To prove the claim, let $m_0\in \mathbb{N}$ denote the rank of the projection $\Delta(p)\in K(H_4)$. Since $H_3$ is infinite-dimensional, we can find a natural $n$ with $n>m_0$ and mutually orthogonal minimal projections $e_2,\ldots, e_{n}$ such that $p+e_1\perp e_j$ for all $j=2,\ldots,n$.\smallskip

We next apply Theorem \ref{t bi spherical set K(H)} to the element $\displaystyle a = p + \sum_{j=1}^{n} \frac12 e_j $. Let us write $\displaystyle q_{n}=\sum_{j=1}^{n} e_j$.
Clearly, $q_n$ is a projection in $K(H_3)$ with $q_n\perp p$, and since $r_{_{B(H_3)}}(a)= p + \displaystyle\sum_{j=1}^{n} e_j = p+ q_{n}$, we have $$Sph^+_{K(H_3)} \left( Sph^+_{K(H_3)}(a) \right) =\left\{ b\in S(K(H_3)^+) : \!\! \begin{array}{c}
    s_{_{K(H_3)}} (a) = p \leq s_{_{K(H_3)}} (b),  \hbox{ and }\\
      \textbf{1}-p - q_n \leq \textbf{1}-r_{_{B(H_3)}}(b)
  \end{array}\!\!
 \right\}, $$ $$ =\left\{ b\in S(K(H_3)^+) : \!\! \begin{array}{c}
    s_{_{K(H_3)}} (a) = p \leq s_{_{K(H_3)}} (b),  \hbox{ and }\\
      b\leq p + q_n
  \end{array}\!\!
 \right\} $$ $$= p+\left\{ x\in \mathcal{B}_{_{K(H_3)^+}} :
  p \perp x \leq  q_n\right\} = p + \mathcal{B}_{_{q_n K(H_3)^+ q_n }}, $$
and the set $\mathcal{B}_{_{q_n K(H_3)^+ q_n }}$ can be C$^*$-isometrically identified with $\mathcal{B}_{B(\ell_2^{n})^+}$.\smallskip

Clearly, the restriction of $\Delta$, to $Sph^+_{K(H_3)}\left( Sph^+_{K(H_3)}(a) \right)$ is a surjective isometry from this set onto $Sph^+_{K(H_4)} \left( Sph^+_{K(H_4)}(\Delta(a)) \right)$. Similarly, by Theorem \ref{t bi spherical set K(H)}, we have $$Sph^+_{K(H_4)} \left( Sph^+_{K(H_4)}(\Delta(a)) \right) = s_{_{K(H_4)}} (\Delta(a)) + \mathcal{B}_{_{ \widehat{q} K(H_4)^+ \widehat{q} }},$$ where $\widehat{q} = r_{_{B(H_4)}} (\Delta(a)) - s_{_{K(H_4)}} (\Delta(a))\in B(H_4)$ and the set $\mathcal{B}_{_{ \widehat{q} K(H_4)^+ \widehat{q} }}$ can be C$^*$-isometrically identified with $\mathcal{B}_{B(H)^+}$, where $H = \widehat{q}(H_4)$ is a complex Hilbert space whose dimension coincides with the rank of the projection $\widehat{q}$. Since every translation, $x\mapsto \tau_z (x) = z+x,$ is a surjective isometry, we can define a surjective isometry $\Delta_a : \mathcal{B}_{B(\ell_2^{n})^+} \to \mathcal{B}_{B(H)^+}$ making the following diagram commutative
\[\begin{tikzcd} Sph^+_{K(H_3)}\left( Sph^+_{K(H_3)}(a) \right) \arrow{r}{\Delta} \arrow[transform canvas={xshift=0.3ex},-]{d} \arrow[transform canvas={xshift=-0.4ex},-]{d} &  Sph^+_{K(H_4)}\left( Sph^+_{K(H_4)}(\Delta(a)) \right) \arrow[transform canvas={xshift=0.3ex},-]{d} \arrow[transform canvas={xshift=-0.4ex},-]{d} \\
 p + \mathcal{B}_{_{q_n K(H_3)^+ q_n }}  \arrow[swap]{d}{\tau_{-p}} & s_{_{K(H_4)}} (\Delta(a)) + \mathcal{B}_{_{ \widehat{q} K(H_4)^+ \widehat{q} }} \\
  \mathcal{B}_{_{q_n K(H_3)^+ q_n }}\cong \mathcal{B}_{B(\ell_2^{n})^+} \arrow{r}{\Delta_a} &  \mathcal{B}_{_{ \widehat{q} K(H_4)^+ \widehat{q} }} \cong \mathcal{B}_{B(H)^+} \arrow[u, "\tau_{s_{_{K(H_4)}} (\Delta(a))}"]
\end{tikzcd}
\]
Actually, $\mathcal{B}_{_{ \widehat{q} K(H_4)^+ \widehat{q} }}$ identifies with the orthogonal to $s_{_{K(H_4)}}(\Delta(a))$ inside the space $r_{_{B(H_4)}}(\Delta(a)) \ K(H_4) \ r_{_{B(H_4)}}(\Delta(a))$.\smallskip

Take a projection $p+r$ in $Sph^+_{K(H_3)}\left( Sph^+_{K(H_3)}(a) \right)$ (clearly $r$ can be any projection in $K(H_3)$ with $ r\leq  q_n$). We know from Corollary \ref{c first consequence K(H)} that $\Delta(p+r)$ is a projection in $Sph^+_{K(H_4)} \left( Sph^+_{K(H_4)}(\Delta(a)) \right)$, and consequently $$\Delta_a (r)=\Delta(p+r) - s_{_{K(H_4)}} (\Delta(a))$$ must be a projection. We have therefore shown that the map $\Delta_a$ above is a surjective isometry mapping projections to projections.\smallskip

We deduce from Lemma \ref{l surjective isometries between balls of positive} that dim$(H)=n,$ and by the same lemma there exists a complex linear (unital) Jordan $^*$-isomorphism $$T_a: q_n K(H_3) q_n\cong B(\ell_2^{n}) \to \widehat{q} K(H_4)^+ \widehat{q} \cong B(\ell_2^{n})$$ satisfying one of the next statements:\begin{enumerate}[$(1)$] \item $\Delta_a (x) = T_a (x),$ for all $x\in \mathcal{B}_{_{q_n K(H_3)^+ q_n }}$;
\item $\Delta_a (x) =\textbf{1}_{\widehat{q}}- T_a (x),$ for all $x\in \mathcal{B}_{_{q_n K(H_3)^+ q_n }}$, where $\textbf{1}_{\widehat{q}} = r_{_{B(H_4)}} (\Delta(a)) - s_{_{K(H_4)}} (\Delta(a))$ is the unit of $\widehat{q} K(H_4)^+ \widehat{q} \cong B(H)$.
\end{enumerate}

We claim that case $(2)$ is impossible. Actually, if case $(2)$ holds, then $$\Delta(p) = s_{_{K(H_4)}} (\Delta(a)) + \Delta_a (0)= s_{_{K(H_4)}} (\Delta(a)) + \left(r_{_{B(H_4)}} (\Delta(a)) - s_{_{K(H_4)}} (\Delta(a))\right) - T_a (0) $$ $$=s_{_{K(H_4)}} (\Delta(a)) + \left(r_{_{B(H_4)}} (\Delta(a)) - s_{_{K(H_4)}} (\Delta(a))\right),$$ where $\left(r_{_{B(H_4)}} (\Delta(a)) - s_{_{K(H_4)}} (\Delta(a))\right)$ and $s_{_{K(H_4)}} (\Delta(a)) $ are orthogonal, and the rank of $\left(r_{_{B(H_4)}} (\Delta(a)) - s_{_{K(H_4)}} (\Delta(a))\right)$ is precisely the dimension of $H$ which is $n$. This shows that $\Delta(p)$ has rank bigger than or equal to $n+1> m_0,$ which is impossible because $m_0$ is the rank of $\Delta (p)$.\smallskip

Since case $(1)$ holds, we have $$\Delta(p+e_1) = s_{_{K(H_4)}} (\Delta(a)) + T_a (e_1)\geq s_{_{K(H_4)}} (\Delta(a)) = \Delta(p),$$ because $T_a (e_1)$ is a non-zero projection and $T_a (e_1)\perp s_{_{K(H_4)}} (\Delta(a)) $. This proves \eqref{eq projection and some minimal orthogonal}. We have also proved that $$s_{_{K(H_4)}} (\Delta(a)) = \Delta (p), \hbox{ and } \Delta (p+ q_n ) = r_{_{B(H_4)}} (\Delta(a)).$$

Now, let $p,q\in \mathcal{P}roj(K(H_3))^*$ with $p\leq q$. In our context we can find mutually orthogonal minimal projections $e_1,\ldots, e_m$ in $K(H_3)$ satisfying $\displaystyle q = p + \sum_{j=1}^m e_j$. Applying \eqref{eq projection and some minimal orthogonal} a finite number of steps we get $$ \Delta (p) \leq \Delta (p+e_1)\leq \ldots \leq \Delta \left(p + \sum_{j=1}^m e_j\right)= \Delta (q).$$

Take now $p,q\in\mathcal{P}roj(K(H_3))^*$ with $p q =0$. Under these hypothesis, Lemma \ref{l orthomorphism preserves orthogonality} assures that $\Delta (p) \Delta (q) =0$.\smallskip

$(b)$ Let us apply the arguments in the proof of $(a)$ to the element $\displaystyle a = p_1 + \sum_{j=2}^n \frac12 p_j$. Let $\displaystyle q_{n-1}= \sum_{j=2}^n p_j$ and $\widehat{q}=\Delta (q_{n-1}) = r_{_{B(H_4)}} (\Delta(a)) - s_{_{K(H_4)}} (\Delta(a))$. We deduce from the above arguments the existence of a surjective isometry $$\Delta_a : \mathcal{B}_{_{q_{n-1} K(H_3)^+ q_{n-1} }}\cong\mathcal{B}_{B(\ell_2^{n-1})^+} \to \mathcal{B}_{_{ \widehat{q} K(H_4)^+ \widehat{q} }} \cong \mathcal{B}_{B(\ell_2^{n-1})^+}$$ making the following diagram commutative
\[\begin{tikzcd} Sph^+_{K(H_3)}\left( Sph^+_{K(H_3)}(a) \right) \arrow{r}{\Delta} \arrow[transform canvas={xshift=0.3ex},-]{d} \arrow[transform canvas={xshift=-0.4ex},-]{d} &  Sph^+_{K(H_4)}\left( Sph^+_{K(H_4)}(\Delta(a)) \right) \arrow[transform canvas={xshift=0.3ex},-]{d} \arrow[transform canvas={xshift=-0.4ex},-]{d} \\
 p_1 + \mathcal{B}_{_{q_{n-1} K(H_3)^+ q_{n-1} }}  \arrow[swap]{d}{\tau_{-p}} & \Delta(p_1) + \mathcal{B}_{_{ \widehat{q} K(H_4)^+ \widehat{q} }} \\
  \mathcal{B}_{_{q_{n-1} K(H_3)^+ q_{n-1} }}\cong \mathcal{B}_{B(\ell_2^{n-1})^+} \arrow{r}{\Delta_a} &  \mathcal{B}_{_{ \widehat{q} K(H_4)^+ \widehat{q} }} \cong \mathcal{B}_{B(H)^+} \arrow[u, "\tau_{\Delta(p_1)}"]
\end{tikzcd}
\]

Since, by $(a),$ $\Delta|_{\mathcal{P}roj(K(H_3))^*}$ is an order automorphism, the reasonings in $(a)$, and Lemma \ref{l surjective isometries between balls of positive} prove the existence of a complex linear (unital) Jordan $^*$-iso-morphism $T_a:B(\ell_2^{n-1})\cong q_{n-1} K(H_3) q_{n-1} \to B(\ell_2^{n-1})\cong \widehat{q} K(H_4) \widehat{q}$ satisfying $$\Delta_a (x) = T_a (x), \hbox{ for all $x\in \mathcal{B}_{B(\ell_2^{n-1})^+}\cong \mathcal{B}_{_{q_{n-1} K(H_3)^+ q_{n-1} }}$.}$$

Pick $j\in \{2,\ldots,n\}$. Since $\Delta|_{\mathcal{P}roj(K(H_3))^*}$ is an order automorphism and preserves orthogonality, the elements $\Delta(p_1),$ $\Delta(p_j),$ and $\Delta(p_1+p_j)$ are non-trivial projections in $K(H_3)$, $\Delta(p_1)$ and $\Delta(p_j)$ are minimal, $\Delta(p_1)\perp \Delta(p_j),$ $\Delta(p_1+p_j)$ is a rank-2 projection, and $\Delta(p_1+p_j) \geq \Delta(p_j)$. We also know that $p_j$ lies in $\mathcal{B}_{_{q_{n-1} K(H_3)^+ q_{n-1} }}$, $T_a (p_j)$ is a minimal projection, $T_a (p_j)\perp \Delta(p_1)$, and $\Delta(p_1+p_j) =\Delta(p_1) +T_a (p_j)$. By applying that $\Delta (p_1) \perp \Delta(p_j)$ we get $$\Delta(p_j) = \Delta(p_1+p_j) \Delta(p_j) = (\Delta(p_1) +T_a (p_j)) \Delta (p_j) = T_a (p_j) \Delta (p_j).$$ The minimality of $T_a (p_j)$ and $\Delta (p_j)$ assures that $T_a (p_j) =  \Delta (p_j).$\smallskip

Finally, given $1=\lambda_1\geq \lambda_2,\ldots, \lambda_n\geq 0$ the element $\displaystyle \sum_{j=1}^{n} \lambda_j p_j = p_1 + \sum_{j=2}^{n} \lambda_j p_j $ lies in the set $Sph^+_{K(H_3)}\left( Sph^+_{K(H_3)}(a) \right)$ and hence $$\Delta\left(\sum_{j=1}^{n} \lambda_j p_j \right) = \Delta(p_1) + \Delta_a\left(\sum_{j=2}^{n} \lambda_j p_j \right) = \Delta(p_1) + T_a\left(\sum_{j=2}^{n} \lambda_j p_j \right) $$ $$= \Delta(p_1) + \sum_{j=2}^{n} \lambda_j T_a\left( p_j \right) =  \Delta(p_1) + \sum_{j=2}^{n} \lambda_j \Delta \left( p_j \right),$$ which finishes the proof of $(b)$.
\end{proof}

Our next corollary is a first consequence of the previous proposition.

\begin{corollary}\label{c linear maps between spheres of Kell2 isometric on the sphere} Let $H_3$ and $H_4$ be separable complex Hilbert spaces. Let us assume that $H_3$ is infinite-dimensional. If $T: K(H_3)\to K(H_4)$ is a bounded \linebreak {\rm(}complex{\rm)} linear mapping such that $T(S(K(H_3)^+)) = S(K(H_4)^+)$ and $T|_{S(K(H_3)^+)} : S(K(H_3)^+) \to  S(K(H_4)^+)$ is a surjective isometry, then $T$ is a $^*$-isomorphism or a $^*$-anti-isomorphism.
\end{corollary}

\begin{proof} Let $T: K(H_3)\to K(H_4)$ be a a bounded linear map satisfying the hypothesis of the corollary. We observe that $T$ must be bijective by hypothesis.\smallskip

We observe that $T(\mathcal{P}roj(K(H_3)))= \mathcal{P}roj(K(H_4)) $ (see Corollary \ref{c first consequence K(H)}), and by Proposition \ref{p second consequence K(H)}, $T$ also preserves order among projections. In particular $T(p) T(q) =0$ for every $p,q \in \mathcal{P}roj(K(H_3))^*$ with $p q =0$ (just observe that the sum of two projections is a projection if and only if they are orthogonal), and thus $T(a^2) = T(a)^2$ and $T(a)^* = T(a),$ whenever $a$ is a finite real linear combination of mutually orthogonal minimal projections in $K(H_3)$. The continuity of $T$ and the norm density in $K(H_3)_{sa}$ of elements which are finite real linear combination of mutually orthogonal minimal projections in $K(H_3)$, imply that $T$ is a Jordan $^*$-isomorphism. The rest is clear from \cite[Corolary 11]{Kad51} because $B(H_3)$ is a factor.
\end{proof}

In the main theorem of this section we extend surjective isometries of the form $\Delta : S(K(H_3)^+)\to S(K(H_4)^+)$. In the proof we shall employ a technique based on the study on the linearity of ``physical states'' on $K(H)$ developed by J.F. Aarnes in \cite{Aarnes70}. We recall that a \emph{physical state} or a \emph{quasi-state} on a C$^*$-algebra $A$ is a function $\rho: A_{sa}\to \mathbb{R}$ whose restriction to each singly generated subalgebra of $A_{sa}$ is a positive linear functional and $$\sup \{ \rho(a) : a \in \mathcal{B}_{_{A^+}}\}= 1.$$ As remarked by Aarnes in \cite[page 603]{Aarnes70}, ``It is far from evident that a physical state on $A$ must be (real) linear on $A_{sa}$'', however, under favorable hypothesis, linearity is automatic and not an extra assumption.

\begin{theorem}\label{t Nagy for K(ell2)} Let $H_3$ and $H_4$ be separable complex Hilbert spaces. Let us assume that $H_3$ is infinite-dimensional. Let $\Delta : S(K(H_3)^+)\to S(K(H_4)^+)$ be a surjective isometry. Then there exists a surjective complex linear isometry $T: K(H_3)\to K(H_4)$ satisfying $T(x) = \Delta(x)$ for all $x\in S(K(H_3)^+)$. We can further conclude that $T$ is a $^*$-isomorphism or a $^*$-anti-isomorphism.
\end{theorem}

\begin{proof} Let $a$ be an element in $S(K(H_3)^+)$, and let us consider the spectral resolution of $a$ in the form $\displaystyle a=\sum_{n=1}^{\infty} \lambda_n p_n,$ where $(\lambda_n)_n$ is a decreasing sequence in $\mathbb{R}_0^+$ converging to zero, $\lambda_1=1$, and $\{ p_n : n\in \mathbb{N}\}$ is a family of mutually orthogonal minimal projections in $K(H_3)$. Applying Proposition \ref{p second consequence K(H)}$(a)$ we deduce that $\{ \Delta(p_n) : n\in \mathbb{N}\}$ is a family of mutually orthogonal minimal projections in $K(H_4)$. Having in mind that orthogonal elements are geometrically $M$-orthogonal, it can be easily deduced that the series $\displaystyle \sum_{n=1}^{\infty} \lambda_n \Delta(p_n)$ is norm convergent. Furthermore, since by Proposition \ref{p second consequence K(H)}$(b)$ and the hypothesis we have $$\left\| \Delta(a) - \sum_{n=1}^{m} \lambda_n \Delta(p_n) \right\| =\left\| \Delta(a) - \Delta\left( \sum_{n=1}^{m} \lambda_n p_n \right) \right\| = \left\| a -  \sum_{n=1}^{m} \lambda_n p_n \right\| = \lambda_{m+1}, $$ it follows that \begin{equation}\label{eq Delta preserves spectral resolution for positive} \Delta(a) = \Delta\left( \sum_{n=1}^{\infty} \lambda_n p_n \right) = \sum_{n=1}^{\infty} \lambda_n \Delta(p_n).
\end{equation} Combining \eqref{eq Delta preserves spectral resolution for positive} and Proposition \ref{p second consequence K(H)}$(a)$ we can see that \begin{equation}\label{eq Delta preserves orthogonality} a\perp b \hbox{ in } S(K(\ell_2)^+) \Rightarrow \Delta (a) \perp \Delta (b).
\end{equation}

Every element $b$ in $K(H_3)_{sa}$ writes uniquely in the form $b = b^+ - b^-$, where $b^+,b^-$ are orthogonal positive elements in $K(H_3)$. Having this property in mind, we define a mapping $T: K(H_3)_{sa} \to K(H_4)_{sa}$ given by $$T(b) := \|b^+\| \Delta \left(\frac{b^+}{\|b^+\|}\right) - \|b^-\| \Delta \left(\frac{b^-}{\|b^-\|}\right), \ \ \hbox{if } \|b^+\| \ \|b^-\| \neq 0,$$
$$T(b) := \|b^+\| \Delta \left(\frac{b^+}{\|b^+\|}\right), \ \ \hbox{if } \|b^+\| \neq 0, b^-=0,$$
$$T(b) := \|b^-\| \Delta \left(\frac{b^-}{\|b^-\|}\right), \ \ \hbox{if } \|b^-\| \neq 0, b^+=0, \hbox{ and } T(0) =0.$$

It follows from definition that \begin{equation}\label{eq boundedness of T}\|T (b)\| \leq \|b^+\| + \|b^-\|\leq 2 \|b\|.
\end{equation}

For each positive functional $\phi \in \mathcal{B}_{_{(K(H_4)^*)^+}}$ we set $T_{\phi} := \phi \circ T : K(H_3)_{sa}\to \mathbb{R}$, $T_{\phi} (x) =\phi (T(x))$. We claim that $T_{\phi}$ is a positive multiple of a physical state. Namely, it follows from \eqref{eq boundedness of T} that $\sup \{ |T_{\phi} (a)| : a\in \mathcal{B}_{_{A^+}} \}\leq 2$. Therefore, we only have to show that the restriction of $T_{\phi}$ to each singly generated subalgebra of $K(H_3)_{sa}$ is linear.\smallskip

Let $b$ be an element in $K(H_3)_{sa}$. We shall distinguish two cases.\smallskip

Case $(a)$: $b$ has finite spectrum. In this case, $b$ is a finite rank operator and $\displaystyle b=\sum_{n=1}^{m} \mu_n p_n,$ where $\mu_1,\ldots, \mu_m\in \mathbb{R}\backslash\{0\}$, and $\{ p_n : n=1,\ldots,m\}$ is a family of mutually orthogonal minimal projections in $K(H_3)$. Elements $x,y$ in the subalgebra of $K(H_3)_{sa}$ generated by $b$ can be written in the form $\displaystyle x=\sum_{n=1}^{m} x(n) p_n,$ and $\displaystyle y=\sum_{n=1}^{m} y(n) p_n,$ where $x(n), y(n)\in \mathbb{R}.$ Let us set $\Theta_x^+=\{n\in \{1,\ldots,m\} : x(n) \geq 0\}$ and $\Theta_x^-=\{n\in \{1,\ldots,m\} : x(n) < 0\}.$
%and similarly for $\Theta_y^+, \Theta_y^-$, $\Theta_{x+y}^+$ and $\Theta_{x+y}^-$.
Suppose that $x^+,x^-\neq 0$. By applying the definition of $T$ we obtain $$ T(x)  = \|x^+\| \Delta \left(\frac{b^+}{\|b^+\|}\right) - \|x^-\| \Delta \left(\frac{b^-}{\|b^-\|}\right)$$ $$ = \|x^+\| \Delta \left(\sum_{n\in \Theta_x^+} \frac{x(n)}{\|x^+\|} p_n \right) - \|x^-\| \Delta \left(\sum_{n\in \Theta_x^-} \frac{-x(n)}{\|x^-\|} p_n \right) $$ $$= \|x^+\| \sum_{n\in \Theta_x^+} \frac{x(n)}{\|x^+\|} \Delta \left(p_n \right) - \|x^-\| \sum_{n\in \Theta_x^-} \frac{-x(n)}{\|x^-\|} \Delta \left(p_n \right) = \sum_{n=1}^{m} x(n) \Delta(p_n),$$ where the penultimate equality follows from Proposition \ref{p second consequence K(H)}$(b)$. In the remaining cases (i.e. $\|x^+\| \|x^-\| =0$) we also have $\displaystyle T(x) = \sum_{n=1}^{m} x(n) \Delta(p_n)$. Since similar conclusions hold for $y$, $x+y$ and $\alpha x$  with $\alpha\in \mathbb{R}$, we deduce that $$ T(x+y ) = \sum_{n=1}^{m} (x(n)+y(n)) \Delta(p_n) = \sum_{n=1}^{m} x(n) \Delta(p_n) + \sum_{n=1}^{m} y(n) \Delta(p_n) = T(x) + T(y),$$ and $$T(\alpha x) = \sum_{n=1}^{m} (\alpha x)(n)  \Delta(p_n) = \alpha \sum_{n=1}^{m} x(n) \Delta(p_n) = \alpha T(x),$$ which shows that $T$ is linear on the subalgebra generated by $b$.\smallskip

Case $(b)$: $b$ has infinite spectrum. In this case, $\displaystyle b=\sum_{n=1}^{\infty} \lambda_n p_n,$ where $(\lambda_n)_n$ is a decreasing sequence in $\mathbb{R}\backslash\{0\}$ converging to zero and $\{ p_n : n\in \mathbb{N}\}$ is a family of mutually orthogonal minimal projections in $K(H_3)$. Elements $x$ and $y$ in the subalgebra of $K(H_3)_{sa}$ generated by $b$ can be written in the form $\displaystyle x=\sum_{n=1}^{\infty} x(n) p_n,$ and $\displaystyle y=\sum_{n=1}^{\infty} y(n) p_n,$ where $(x(n))$ and $(y(n))$ are null sequences in $\mathbb{R}.$ Keeping in mind the notation employed in the previous paragraph we deduce that if $x^+,x^-\neq 0$ we have $$ T(x)  = \|x^+\| \Delta \left(\frac{b^+}{\|b^+\|}\right) - \|x^-\| \Delta \left(\frac{b^-}{\|b^-\|}\right)$$ $$ = \|x^+\| \Delta \left(\sum_{n\in \Theta_x^+} \frac{x(n)}{\|x^+\|} p_n \right) - \|x^-\| \Delta \left(\sum_{n\in \Theta_x^-} \frac{-x(n)}{\|x^-\|} p_n \right) = \hbox{(by \eqref{eq Delta preserves spectral resolution for positive})}$$ $$= \|x^+\| \sum_{n\in \Theta_x^+} \frac{x(n)}{\|x^+\|} \Delta \left(p_n \right) - \|x^-\| \sum_{n\in \Theta_x^-} \frac{-x(n)}{\|x^-\|} \Delta \left(p_n \right) = \sum_{n=1}^{\infty} x(n) \Delta(p_n).$$ In the remaining cases the identity \begin{equation}\label{eq T preserves spectral resolutions too} T(x) = \sum_{n=1}^{\infty} x(n) \Delta(p_n)
 \end{equation} also holds. It is therefore clear that $T$ is linear on the subalgebra generated by $b$.\smallskip

We have therefore proved that, $T_{\phi}: K(H_3)_{sa}\to \mathbb{R}$ is a positive multiple of a physical state for every $\phi \in \mathcal{B}_{_{(K(H_4)^*)^+}}$. Applying \cite[Corollary 2]{Aarnes70} to the complex linear extension of $T_{\phi}$ from $K(H_3)$ to $\mathbb{C}$ it follows that $$\phi (T(x+y)) = T_{\phi} (x+y) = T_{\phi} (x) + T_{\phi} (y) = \phi (T(x) + T(y)),$$ and $$\phi (T(\alpha x) ) = T_{\phi} (\alpha x) = \alpha T_{\phi} (x) =  \phi (\alpha T(x)),$$ for all $x,y\in  K(H_3)_{sa},$ $\alpha\in \mathbb{R}$, and $\phi \in \mathcal{B}_{_{(K(H_4)^*)^+}}$. Since functionals in $\mathcal{B}_{_{(K(H_4)^*)^+}}$ separate the points in $K(H_4)_{sa}$, we deduce that $T: K(H_3)_{sa}\to K(H_4)_{sa}$ is real linear. We denote by the same symbol $T$ the complex linear extension of $T$ from $K(H_3)$ to $K(H_4)$. We have obtained a complex linear map $T :  K(H_3)\to K(H_4)$ satisfying $T(a) = \Delta(a)$ for all $a\in S( K(H_3)^+)$ (compare \eqref{eq Delta preserves spectral resolution for positive} and \eqref{eq T preserves spectral resolutions too}). Corollary \ref{c linear maps between spheres of Kell2 isometric on the sphere} assures that $T :  K(H_3)\to K(H_4)$ is an isometric $^*$-isomorphism or $^*$-anti-isomorphism. \end{proof}

\medskip\medskip

\textbf{Acknowledgements} Author partially supported by the Spanish Ministry of Economy and Competitiveness (MINECO) and European Regional Development Fund project no. MTM2014-58984-P and Junta de Andaluc\'{\i}a grant FQM375.

\bibliographystyle{amsplain}

\end{document}